\documentclass[reqno]{amsart}
\usepackage[T1]{fontenc}
\usepackage[utf8]{inputenc}
\usepackage{color}
\usepackage{amstext}
\usepackage{amsthm}
\usepackage{amssymb}
\usepackage{graphicx}
\usepackage[unicode=true,pdfusetitle,
 bookmarks=true,bookmarksnumbered=false,bookmarksopen=false,
 breaklinks=false,pdfborder={0 0 0},backref=false,colorlinks=true]
 {hyperref}
\hypersetup{
 citecolor=blue}

\makeatletter
\numberwithin{equation}{section}
\numberwithin{figure}{section}
\theoremstyle{plain}
\newtheorem*{cor*}{\protect\corollaryname}
\theoremstyle{plain}
\newtheorem{thm}{\protect\theoremname}[section]
\theoremstyle{definition}
\newtheorem{defn}[thm]{\protect\definitionname}
\theoremstyle{remark}
\newtheorem{rem}[thm]{\protect\remarkname}
\theoremstyle{plain}
\newtheorem{prop}[thm]{\protect\propositionname}
\theoremstyle{plain}
\newtheorem{lem}[thm]{\protect\lemmaname}
\theoremstyle{plain}
\newtheorem{cor}[thm]{\protect\corollaryname}
\theoremstyle{plain}

\usepackage[bottom=1in]{geometry}
\usepackage{amsmath}
\usepackage{amsthm}

\numberwithin{equation}{section}
\numberwithin{figure}{section}
\usepackage{enumitem}		
 \let\footnote=\endnote
\@ifundefined{lettrine}{\usepackage{lettrine}}{}

\theoremstyle{definition}
\newtheorem{thmx}{Theorem}

\def\a{\alpha}
\def\A{\mathcal{A}}
\def\b{\beta}
\def\B{\mathcal{B}}
\def\s{\sigma}

\def\L{\Lambda}

\def\R{\mathbb{R}}

\def\L{\mathcal{L}}

\def\CC{\mathcal{C}}
\def\N{\mathbb{N}}
\def\Z{\mathbb{Z}}

\def\ep{\varepsilon}
\def\vps{\varphi^s}
\def\Md{M_{d \times d}(\R)}

\def\I{\mathsf{I}}
\def\J{\mathsf{J}}
\def\K{\mathsf{K}}

\def\W{\mathcal{W}}
\def\WW{\mathbb{W}}

\newcommand{\Wloc}{\mathcal{W}_{\text{loc}}}

\def\vp{\varphi}
\def\hol{H\"older }
\def\wt{\widetilde}
\def\Sig{\Sigma_T}
\def\M{\mathcal{M}}

\def\U{\mathcal{U}}

\def\glr{\text{GL}_d(\R)}
\def\gltwo{\text{GL}_2(\R)}          

\usepackage{float}

\address{Department of Mathematics,
The Ohio State University, Columbus, OH 43210, USA} 
\email{call.119@osu.edu}

\address{Department of Mathematics, The University of Chicago, Chicago, IL 60637, USA} 
\email{kihopark@math.uchicago.edu}

\makeatother

  \providecommand{\corollaryname}{Corollary}
  \providecommand{\definitionname}{Definition}
  \providecommand{\lemmaname}{Lemma}
  \providecommand{\propositionname}{Proposition}
  \providecommand{\remarkname}{Remark}
  \providecommand{\theoremname}{Theorem}
\providecommand{\theoremname}{Theorem}

\usepackage{ulem}

\newcommand{\comment}[1]{{}}
\begin{document}

\title[The $K$-property for Subadditive Equilibrium States]{The $K$-property for Subadditive Equilibrium States}
\author{Benjamin Call, Kiho Park}

\thanks{B.C.\ is partially supported by NSF grant DMS-$1461163$.}
\date{\today}

\begin{abstract}
By generalizing Ledrappier's criterion \cite{Ledrappier1977K} for the $K$-property of equilibrium states, we extend the criterion to subadditive potentials. We apply this result to the singular value potentials of matrix cocycles, and show that equilibrium states of large classes of singular value potentials have the $K$-property. 
\end{abstract}

\maketitle

\section{Introduction}

Given a continuous potential $\vp \colon X \to \R$ over a dynamical system $(X,f)$, its pressure may be defined as
$$P(\vp) = \sup\limits_{\mu \in \M(f)} \Big\{h_\mu(f)+\int \vp \,d\mu\Big\}$$
and we call the $f$-invariant measures achieving the supremum \textit{equilibrium states}. These play an important role in the study of the dynamical system $(X,f)$. Provided the entropy map $\mu\mapsto h_\mu(f)$ is upper semicontinuous, these equilibrium states exist. However, without further assumptions on the potential $\vp$ or the base dynamical system $(X,f)$, their uniqueness is not guaranteed.   

On the other hand, in his fundamental work, Bowen \cite{bowen1974some} established the following result that guarantees the existence and uniqueness of equilibrium states. Given a potential with a regularity condition, later named the Bowen property, over an expansive dynamical system with specification, the system has a unique equilibrium state; see Proposition \ref{prop: bowen thm} and \ref{prop: bowen general} for more details. Such unique equilibrium states are now well-studied with various known constructions as well as strong ergodic and statistical properties; see \cite{bowen1974some, bowen1975ergodic, ratner1973central, ruelle1976measure, parry1990zeta}. An important example that fits into this framework consists of \hol potentials over uniformly hyperbolic systems. 

Since then, the theory has been extended in mainly two different directions. One direction aims to relax the uniform hyperbolicity of the base dynamics; see for instance \cite{Knieper:1998ht, CT, Burns:2018ed,chen2020unique}. The other aims to relax and generalize the assumptions on the potential. In particular, much attention was recently brought to the subadditive generalization of thermodynamic formalism due to its applications to the dimension theory of fractals; see for instance \cite{falconer1988subadditive, zhang1997dynamical, chen2010dimension, ban2010dimensions, feng2014non, barany2017hausdorff} and references therewithin. 
In this paper, we pursue the latter generalization and study ergodic properties of the subadditive equilibrium states.

Denoting a mixing subshift of finite type by $(\Sig,\s)$ and the full shift by $(\Sigma,\sigma)$, consider a sequence of continuous functions $\Phi = \{\log \vp_n\}_{n \in \N}$ on $\Sig$. We say $\Phi$ is \textit{subadditive} if
\begin{equation}\label{eq: subadditive 1}
\log\vp_{m+n} \leq \log \vp_m +\log\vp_n \circ \s^m
\end{equation}
for all $m,n \in \N$. Associated to the subadditive potential $\Phi$ is the subadditive pressure $P(\Phi)$ introduced by Cao, Feng, and Huang \cite{cao2008thermodynamic} by generalizing the usual definition of the pressure in thermodynamic formalism; see Section \ref{sec: prelim}. In particular, the subadditive pressure $P(\Phi)$ satisfies the subadditive variational principle \cite{cao2008thermodynamic}:
\begin{equation}\label{eq: var prin}
P(\Phi) = \sup\limits_{\mu \in \M(\s)} \Big\{ h_\mu(\s) + \lim\limits_{n \to \infty}\frac{1}{n} \int \log\vp_n\,d\mu \Big\}.
\end{equation}
Any $\s$-invariant measure $\mu \in \M(\s)$ achieving the supremum in \eqref{eq: var prin} is called an \textit{equilibrium state} of $\Phi$.

Denoting by $\L$ the set of all admissible words of $\Sig$, we associate to a subadditive potential $\Phi = \{\log \vp_n\}_{n \in \N}$ a function $\wt{\Phi} \colon \L \to \R$ defined by
\begin{equation}\label{eq: wt Phi}
\wt{\Phi}(\I):=\sup\limits_{x \in [\I]} \vp_n(x).
\end{equation} 
We say a subadditive potential $\Phi =\{\log \vp_n\}_{n \in \N}$ is \textit{quasi-multiplicative} if there exists $c>0$ and $k \in \N$ such that for any $\I,\J \in \L$, there exists $\K \in \L$ with $|\K| \leq k$ such that
$$\wt{\Phi}(\I\K\J) \geq c \wt{\Phi}(\I)\wt{\Phi}(\J).$$	

Quasi-multiplicativity may be thought of as follows: given any two words $\I,\J \in \L$ of arbitrary length, we obtain the inequality opposite to subadditivity \eqref{eq: subadditive 1} at a cost of inserting a connecting word $\K \in \L$ of bounded length in between $\I$ and $\J$. Quasi-multiplicativity is a rather mild assumption, and hence it is enjoyed by a large class of subadditive potentials; see Proposition \ref{prop: irred implies qm} and Proposition \ref{prop: typical unique eq}. A particularly important application is that, together with the bounded distortion property, quasi-multiplicativity serves as a sufficient condition to generalize Bowen's theorem on the uniqueness of equilibrium states; see Proposition \ref{prop: unique eq state} for the precise statement.

In this paper we study ergodic properties of unique equilibrium states guaranteed by quasi-multiplicativity and bounded distortion. Morris \cite{morris2018ergodic} recently showed that if the unique equilibrium states associated to a class of matrix cocycles are totally ergodic, then they are actually mixing.
The main result of this paper is similar in flavor to Morris's result. However, we work in a more general class of subadditive potentials, and show that under suitable assumptions, total ergodicity can be promoted to the $K$-property, which is stronger than mixing of all orders and weaker than Bernoulli.

\begin{thmx}\label{thm: A}
Let $\Phi=\{\log\vp_n\}_{n \in \N}$ be a subadditive potential on $\Sig$, and suppose it is quasi-multiplicative and has bounded distortion. Suppose further that the unique equilibrium state $\mu \in \M(\s)$ of $\Phi$ guaranteed by Proposition \ref{prop: unique eq state} is totally ergodic. Then, $\mu$ has the $K$-property.
\end{thmx}

The key result used to prove this theorem is Theorem \ref{thm: Subadditive Kprop result}, which shows that for subadditive equilibrium states, weak mixing is equivalent to the $K$-property under some suitable assumptions, similar to those used in \cite{bowen1974some}. This holds even for non-symbolic systems, and we expect it to be of independent interest.
The remaining results in this paper are obtained by applying Theorem \ref{thm: A} to $\glr$-cocycles, including locally constant cocycles and fiber-bunched cocycles. For any cocycle $\A \colon \Sig \to \glr$, we define its \textit{norm potential} $\Phi_\A = \{\log \vp_{\A,n}\}_{n \in \N}$ as
$$\vp_{\A,n}(x):=\|\A(\s^{n-1}x) \ldots\A(x)\|.$$
From the submultiplicativity of the operator norm $\|\cdot\|$, it is clear that $\Phi_\A$ is subadditive.
The singular value potentials are natural generalizations of the norm potentials; see Section \ref{sec: prelim} for the precise definition.

For irreducible locally constant cocyles $\A \colon \Sigma \to \Md$, it was shown by Feng \cite{feng2009lyapunov} that norm potentials $\Phi_\A$ have unique equilibrium states $\mu_\A \in \M(\s)$ by establishing quasi-multiplicativity. Morris \cite{morris2019necessary} then obtained a characterization for $\mu_\A$ to be mixing under an extra assumption that at least one matrix in the image of $\A$ is invertible. This assumption is automatically met when $\A$ takes values in $\glr$. We partly reformulate \cite[Corollary 3]{morris2019necessary} for $\glr$-cocycles. 

\begin{prop}\cite[Corollary 3]{morris2019necessary}\label{prop: Morris equivalence} Suppose $\A \colon \Sigma \to \glr$ is an irreducible locally constant cocycle. Then the following are equivalent:
\begin{enumerate}
\item The unique equilibrium state $\mu$ is mixing with respect to $\s$.
\item The measure $\mu$ is ergodic with respect to $\s^d$.
\end{enumerate}
\end{prop}

\begin{rem}
In both \cite{morris2018ergodic} and \cite{morris2019necessary}, Morris works with one-sided full shifts only, and so for brevity, we limit the following result to full shifts, though we expect it to hold for shifts of finite type. The generalization to two-sided shifts follows from the theory of natural extensions, which we discuss in Subsection \ref{sec: two sided}.
\end{rem}

By direct application of Theorem \ref{thm: A} to norm potentials of irreducible locally constant cocycles $\A \colon \Sigma \to \glr$, we may add to the list of equivalent conditions in Proposition \ref{prop: Morris equivalence} that $\mu$ has the $K$-property. This improves the result of Morris.

\begin{thmx}\label{thm: B} Let $\A \colon \Sigma \to \glr$ be an irreducible locally constant cocycle. Suppose the unique equilibrium state $\mu_\A \in \M(\s)$ of $\Phi_\A$ satisfies any one of the equivalent conditions from Proposition \ref{prop: Morris equivalence}. Then $\mu_\A$ has the $K$-property.
\end{thmx}

In Conjecture 2 of \cite{morris2018ergodic}, Morris conjectured that the natural extension of every totally ergodic matrix equilibrium state for a certain collection of potentials has the Bernoulli property. Theorem \ref{thm: B} establishes partial progress towards this conjecture for the class of norm potentials of $\glr$ locally constant cocycles, and in Remark \ref{rem: other potentials}, we discuss how Theorem \ref{thm: A} applies to all potentials considered in \cite{morris2018ergodic}. We note that Theorem \ref{thm: B} is related to results of Feng \cite{feng2011equilibrium} where he establishes the $K$-property for some weighted equilibrium states.

The remaining results are concerned with the thermodynamic formalism of $\alpha$-\hol and fiber-bunched cocycles; see Definition \ref{defn: fb}. We denote the space of $\alpha$-\hol and fiber-bunched $\glr$-cocycles by $C^\alpha_b(\Sig,\glr)$. In particular, every locally constant cocycle falls into these categories, while the converse is not true.

Specific to fiber-bunched $\gltwo$-cocycles, in \cite{butler2019thermodynamic} Butler and the second-named author obtain a precise description of the equilibrium states for the norm potentials $\Phi_\A$. In particular, they have a complete characterization for when the norm potentials $\Phi_\A$ of fiber-bunched $\gltwo$-cocycles $\A$ fail to have unique equilibrium states.
By considering all cases depending on the number of equilibrium states, we show that all equilibrium states of fiber-bunched $\gltwo$-cocycles have the $K$-property.

\begin{thmx}\label{thm: C} Let $\A \colon \Sig \to \gltwo$ be a \hol continuous and fiber-bunched cocycle. Every ergodic equilibrium state of $\A$ is $K$ up to a period, that is, $(\Sig,\s,\mu_\A)$ is isomorphic to a $K$-system times a finite rotation. Indeed, the only setting when $\mu_\A$ is not $K$ is when $\A$ can be conjugated to another cocycle 
$$\B(x)= \begin{pmatrix}
0 & a(x) \\
b(x) & 0
\end{pmatrix}$$
such that $\a(x):=\log |a(\s x)b(x)|$ and $\b(x):=\log |b(\s x)a(x)|$ viewed as potentials over $(\Sig,\s^2)$ have the same pressures, but distinct equilibrium states $\mu_1$ and $\mu_2$.
\end{thmx}

Theorem \ref{thm: C} may be thought of as follows:  for fiber-bunched $\gltwo$-cocycles, the specified case is the only case where the ergodic equilibrium states for $\Phi_\A$ fail to be $K$.	

The final result of this paper applies to a large subset of fiber-bunched cocycles and their singular value potentials.
More specifically, Bonatti and Viana \cite{bonatti2004lyapunov} introduced a notion of typical cocycles among $C^\alpha_b(\Sig,\glr)$ and established that typicality serves as a sufficient condition for the simplicity of Lyapunov exponents with respect to any ergodic measures with continuous local product structure. Additionally, they showed that the set of typical cocycles is open and dense in $C^\alpha_b(\Sig,\glr)$; see Definition \ref{defn: typical} for the precise formulation of the typicality assumption.

In \cite[Theorem B]{park2019quasi}, the second-named author shows that for any $s \in [0,\infty)$ the singular value potentials $\Phi_{\A}^s$ of typical cocycles $\A$ have unique equilibrium states $\mu_{\A,s} \in \M(\s)$. By verifying the assumptions in Theorem \ref{thm: A}, we show that such  equilibrium states $\mu_{\A,s}$ have the $K$-property:

\begin{thmx}\label{thm: D} Let $\A \colon\Sig \to \glr$ be a \hol continuous and fiber-bunched cocycle. If $\A$ is typical, then for any $s \in [0,\infty)$ the unique equilibrium state $\mu_{\A,s} \in \M(\s)$ of $\Phi_\A^s$ has the $K$-property.
\end{thmx}
We remark that unlike Theorem \ref{thm: B} and \ref{thm: C} where our results only apply to norm potentials $\Phi_\A$, the typicality assumption in Theorem \ref{thm: D} allows us to apply it to singular value potentials $\Phi_\A^s$ for all $s \in [0,\infty)$.

In Section \ref{sec: prelim}, we introduce and survey relevant preliminary results. In Section \ref{sec: K}, we establish sufficient criteria for the $K$-property in the subadditive setting. Then we prove the main theorems in Section \ref{sec: proof of main thms}.

\section{Preliminaries}\label{sec: prelim}
Throughout, $X$ is a compact metric space, $f : X\to X$ is a homeomorphism, and $\M(f)$ denotes the set of all $f$-invariant probability measures. 

\subsection{Mixing properties of invariant measures}
Although many of the results in this paper are concerned with the $K$-property, we make use of various mixing properties along the way. We give a brief introduction to the ones that we will use, in increasing order of strength. In many cases, there are many equivalent formulations of these definitions which we lack the space to include. For a more comprehensive discussion on various mixing properties, we refer the readers to \cite{Petersen}

\begin{defn}
A measure-preserving transformation $(X,f,\mu)$ is \textit{totally ergodic} if $(X,f^n,\mu)$ is ergodic for all $n \in \N$. 
\end{defn}

\begin{defn}
	We say $\mu \in \M(f)$ is \textit{weakly mixing} if for all measurable subsets $A,B \subseteq X$, there exists $E\subset\N$ with upper density $\bar{d}(E) = 0$ such that for $n\notin E$, 
	$$\lim\limits_{n\to\infty}\mu(f^nA\cap B) = \mu(A)\mu(B).$$ 
\end{defn}

Observe from this definition that if $\mu$ is weak mixing, then it is totally ergodic. The following is a now classical result that we will also make use of periodically.

\begin{prop}\label{prop: weak mixing iff}
$(X,f,\mu)$ is weak mixing if and only if $(X\times X,f\times f,\mu\times\mu)$ is ergodic.
\end{prop}

Finally, we introduce the $K$-property. There are a myriad number of equivalent formulations, for details of which we refer to \cite{CFSS}.

\begin{defn}
	Let $(X,\mathcal{B},\mu)$ be a probability space, and let $f$ be a measure-preserving invertible transformation of $(X,\mathcal{B},\mu)$ 
	The system has the \textit{Kolmogorov property}, or simply the \textit{K-property}, if there exists a sub-$\sigma$-algebra $\mathcal{K} \subset \mathcal{B}$ satisfying $\mathcal{K} \subset f\mathcal{K}$, $\bigvee\limits_{i=0}^\infty f^i\mathcal{K} = \mathcal{B}$, and $\bigcap\limits_{i=0}^\infty f^{-i}\mathcal{K} = \{\emptyset,X\}$.
\end{defn}
An equivalent definition of independent interest is that of \textit{completely positive entropy}.
\begin{prop}
A measure $\mu\in \M(f)$ has the $K$-property if and only if it has completely positive entropy, that is, there are no non-trivial zero entropy factors.
\end{prop}

\begin{rem}
In particular, any $K$-system has positive measure-theoretic entropy.
\end{rem}

We say a system $(X,f,\mu)$ is \textit{Bernoulli} if it is measurably isomorphic to a Bernoulli shift. Without too much difficulty, one can use the above definition to show that every Bernoulli system is $K$.

\subsection{Base dynamical system}
Let $T$ be a $q \times q$ square-matrix with entries in $\{0,1\}$. We define $\Sig \subset \{1,\cdots, q\}^\Z$ to be the set of all bi-infinite sequences of $q$ symbols such that $ij$ is a word for $1\leq i,j\leq q$ if and only if $T_{ij} = 1$. An \textit{admissible word of length $n$} is a word $i_0\ldots i_{n-1}$ with $i_j \in \{1,\ldots,q\}$ such that $T_{i_j,i_{j+1}} = 1$ for all $0 \leq j \leq n-2$.
Let $\L$ be the collection of all admissible words. For $\I\in \L$, we denote its length by $|\I|$.
For each $n \in \N$, let $\L(n) \subset \L$ be the set of all admissible words of length $n$. For any $\I =  i_0\ldots i_{n-1} \in \L(n)$, we define the associated \textit{cylinder} by
$$[\I]=[ i_0\ldots i_{n-1}]:=\{y \in \Sig \colon y_j = i_j \text{ for all } 0 \leq j \leq n-1\}.$$

We endow $\Sig$ with the metric $d$ defined as follows: for $x = (x_i)_{i \in \Z},y = (y_i)_{i \in \Z} \in \Sig$, we have
$$d(x,y)  = 2^{-k},$$ 
where $k$ is the largest integer such that $x_i = y_i$ for all $|i| < k$. Equipped with such a metric, the left shift operator $\s$ becomes a hyperbolic homeomorphism of a compact metric space $\Sig$. Given $x\in\Sig$, the \textit{local stable set} of $x$ is
$$\Wloc^s(x) = \{y\in\Sig\mid x_i = y_i \text{ for } i\geq 0 \}$$
and analogously, the \textit{local unstable set} of $x$ is
$$\Wloc^u(x) = \{y\in\Sig\mid x_i = y_i \text{ for } i\leq 0\}.$$
These local stable and unstable sets extend to define global stable and unstable sets $\W^{s/u}(x)$, respectively, in the standard manner.

Finally, we will always assume that the adjacency matrix $T$ is \textit{primitive}, meaning that there exists $N>0$ such that all entries of $T^N$ are positive. The primitivity of $T$ is equivalent to $(\Sig,\s)$ being topologically mixing.

\subsection{Linear cocycles}

To any $\A \colon \Sig \to \Md$ and $n \in \N$, we define
$$\A^n(x):=\A(\s^{n-1}x)\ldots \A(x).$$
It is clear from the definition that the following \textcolor{black}{\textit{cocycle equation}} holds:
$$\A^{n+m}(x) := \A^n(\s^mx)\A^m(x) \text{ for all } n,m\in \N.$$
When the image of $\A$ is a subset of $\glr$, we define $\A^{0}(\cdot) \equiv I$ and $\A^{-n}(x) : = \big(\A^n(\s^{-n}x)\big)^{-1}$ for $n \in \N$ 
\textcolor{black}{so} that the cocycle equation holds for all $n,m \in \Z$. 

We now introduce two classes of cocycles appearing in Theorems \ref{thm: B}, \ref{thm: C}, and \ref{thm: D}. 
First is the class of locally constant cocycles. A \textit{locally constant cocycle} $\A$ is a cocycle whose generator $\A$ is locally constant. If $\A \colon \Sig \to \Md$ is locally constant, then from the compactness of $\Sig$, there exists $k\in \N$ such that $\A(x)$ depends only on the word $x_{-k}\ldots x_k \in \L(2k+1)$ for every $x = (x_i)_{i \in \Z} \in \Sig$. For any locally constant $\glr$-valued function $\A$ on $\Sig$, there exists a recoding of $\Sig$ to another subshift of finite type 
$\Sigma_{S}$ such that $\A$ is carried to a $\glr$-valued function on $\Sigma_{S}$ depending only on the 0-th entry $x_0$ of $x = (x_i)_{i\in \Z} \in \Sigma_{S}$.

\begin{rem}
For simplicity, we assume that all locally constant cocycles considered in this paper depend only on the 0-th entry.
\end{rem}

The second class consists of fiber-bunched cocycles:

\begin{defn}\label{defn: fb} An $\alpha$-\hol cocycle $\A \in C^\alpha(\Sig,\glr)$ is \textit{fiber-bunched} if for every $x \in \Sig$,
	$$\|\A(x)\| \cdot \|\A(x)^{-1}\| < 2^\alpha.$$
\end{defn}

Clearly, conformal cocycles are fiber-bunched. Moreover, small perturbations of conformal cocycles are also fiber-bunched; in fact, fiber-bunched cocycles may be thought of as nearly conformal cocycles. 
We denote the set of $\alpha$-\hol and fiber-bunched cocycles by
$C^\alpha_b(\Sig,\glr).$ From the definition $C^\alpha_b(\Sig,\glr)$, is an open subset of $C^\alpha(\Sig,\glr)$.

The fiber-bunching assumption is mainly used for the convergence of the \textit{canonical stable/unstable holonomy} $H^{s/u}_{x,y}$: for any $y \in \Wloc^{s/u}(x)$, 
\begin{equation}\label{eq: canonical holonomies}
H^s_{x,y} :=\lim\limits_{n \to \infty} \A^n(y)^{-1}\A^n(x) ~\text{ and }~H^u_{x,y}:= \lim\limits_{n \to -\infty} \A^n(y)^{-1}\A^n(x).
\end{equation}
Moreover, the canonical holonomies vary \hol continuously in the basepoints $x,y\in \Sig$ with $y \in \Wloc^{s/u}(x)$: there exists $C>0$ such that
\begin{equation}\label{eq: hol holder}
\|H^{s/u}_{x,y} - I\| \leq C\cdot d(x,y)^{\alpha}.
\end{equation}
See \cite{kalinin2013cocycles} for further details.

It can be easily checked that the canonical stable holonomies $H^{s}_{x,y}$ satisfy the following properties:
\begin{enumerate}
	\item $H^s_{x,x} = I$ and $H^s_{y,z}\circ H^s_{x,y} = H^s_{x,z}$ for any $y,z \in \Wloc^s(x)$,
	\item $\A(x) = H^s_{\s y,\s x} \circ \A(y) \circ H^s_{x,y}$,
	\item $H^s\colon (x,y) \mapsto H^s_{x,y}$ is continuous.
\end{enumerate}
Likewise, the canonical unstable holonomies $H^u_{x,y}$ satisfy the analogous properties. Using the second property, the canonical holonomies $H^{s/u}$ can be defined for $y \in \W^{s/u}(x)$ as well; i.e., for $y$ not necessarily in $\Wloc^{s/u}(x)$ but belonging to $\W^{s/u}(x)$.


We now formulate the typicality assumption appearing in Theorem \ref{thm: D}. Consider any periodic point $p \in \Sig$ and a homoclinic point $z \in \W^s(p)\cap \W^u(p) \setminus\textcolor{black}{\{p\}}$. We define the \textit{holonomy loop} $\psi_p^z$ as the composition of the unstable holonomy from $p$ to $z$ and the stable holonomy from $z$ to $p$:
$$
\psi_p^z  := H^s_{z,p}\circ H^u_{p,z}.
$$

The following definition is a slight variation of typicality first introduced in \cite{bonatti2004lyapunov}; this version of typicality is identical to the definition which appeared in \cite{park2019quasi}.
\begin{defn}\label{defn: 1-typical}
	Let $\A \in C^\alpha_b(\Sig,\glr)$ be a fiber-bunched cocycle and $H^{s/u}$ be its canonical holonomies. We say that $\A$ is \textit{1-typical} if it satisfies the following two extra conditions:
	\begin{enumerate}
		\item there exists a periodic point $p$ such that $P:=\A^{\text{per}(p)}(p)$ has simple real eigenvalues of distinct norms. Let $\{v_i\}_{1\leq i \leq d}$ be the eigenvectors of $P$.
		\item there exists a homoclinic point $z$ of $p$ such that $\psi_p^z$ twists the \textcolor{black}{eigendirections} of $P$ into general position: 
		for any $1 \leq i,j \leq d$, the image $\psi_{p}^z(v_i)$ does not lie in any hyperplane $\WW_j$ spanned by all eigenvectors of $P$ other than $v_j$. Equivalently, the coefficients $c_{i,j}$ in 
		$$\psi_{p}^z(v_i) = \sum\limits_{1 \leq j \leq d}c_{i,j} v_j,$$
		are nonzero for all $1 \leq i,j \leq d$.
	\end{enumerate}
These two conditions in the above definition are often called \textit{pinching} and \textit{twisting}, respectively. 
\end{defn}

For each $1 \leq t \leq d$, we denote by $\A^{\wedge t}$ the action of $\A$ on the exterior product $(\R^d)^{\wedge t}$.
Then the exterior product cocycles $\A^{\wedge t},~t \in \{1, \ldots, d\}$, also admit stable and unstable holonomies, namely $(H^{s/u})^{\wedge t}$. So, for a 1-typical function $\A$, we consider similar conditions appearing in Definition \ref{defn: 1-typical} on $\A^{\wedge t}$.

\begin{defn}\label{defn: t-typical} Let $\A \in C^\alpha_b(\Sig,\glr)$ be 1-typical.
	For $2 \leq t \leq d-1$, we say $\A$ is \textit{t-typical} if the same points $p,z \in \Sig$ from Definition \ref{defn: 1-typical} satisfy 
	\begin{enumerate}
		\item all the products of $t$ distinct eigenvalues of $P$ are distinct;
		\item the induced map $(\psi_p^z)^{\wedge t} $ on $(\R^d)^{\wedge t}$ satisfies the corresponding twisting condition to that given by Definition \ref{defn: 1-typical} with respect to the eigenvectors $\{v_{i_1}\wedge \ldots \wedge v_{i_t}\}_{1\leq i_1<\ldots<i_t \leq d}$ of $P^{\wedge t}$.
	\end{enumerate}
\end{defn}

\begin{defn}\label{defn: typical}
	We say $\A \in C^\alpha_b(\Sig,\glr)$ is \textit{typical} if $\A$ is $t$-typical for all $1\leq t \leq d-1$.
\end{defn}
\begin{rem}\label{rem: typicality comments}
	Typicality was first introduced by Bonatti and Viana  \cite{bonatti2004lyapunov} for fiber-bunched $\text{SL}_d(\R)$-cocycles as a sufficient condition to guarantee the simplicity of Lyapunov exponents with respect to any ergodic invariant measures with continuous local product structure and full support. They also showed that the set of typical cocycles is open and dense in $C^\alpha_b(\Sig,\text{SL}_d(\R))$ and this property easily generalizes to fiber-bunched $\glr$-cocycles. 

	The pinching and twisting assumptions of typicality are designed to replicate the effects of proximality and strong irreducibility from Furstenberg's theorem \cite{furstenberg1963noncommuting} on positivity of the top Lyapunov exponent.
\end{rem}

\subsection{Thermodynamic Formalism}
In this section, we will introduce some of the key ideas of both additive and subadditive thermodynamic formalism that we will use. For shorthand, we will often refer to the metric
$$d_n(x,y) := \max_{0\leq i\leq n-1}d(f^ix,f^iy)$$
and to \textit{Bowen balls}
$$B_n(x,\ep) := \{y \in X \mid d_n(x,y) \leq \ep\}.$$
For any $n\in\N$ and $\ep > 0$, $E\subset X$ is \textit{$(n,\ep)$-separated} if given $x,y\in E$, $d_n(x,y) \geq \ep$. Such a set is \textit{maximally} $(n,\ep)$-separated if for any $z\notin E$, $d_n(x,z) \leq \ep$ for some $x\in E$.

Whenever we look at at a product space, we take the metric to be the maximum of the distance in each coordinate:
\begin{equation}\label{eq: product metric}
d((x_1,y_1),(x_2,y_2)): = \max\{d(x_1,x_2),d(y_1,y_2)\}.
\end{equation}

As mentioned in the introduction, we are interested in studying mixing properties of unique equilibrium states. One of the ``ideal'' results is the following, a proof of which can be found in \cite[Theorem 4.1]{bowen1975ergodic}. 
\begin{prop}\label{prop: bowen thm}
Let $f : X\to X$ be a transitive Anosov homeomorphism, and let $\vp : X\to\R$ be a H\"older continuous potential. Then $\vp$ has a unique equilibrium state $\mu \in \M(f)$, and $\mu$ is Bernoulli.
\end{prop}

The proof of this relies on the construction of a Markov coding, which in turn establishes the Bernoulli property. In general, however, the existence, uniqueness, as well as the mixing properties of equilibrium states are not well-known. While the existence of equilibrium states is often guaranteed under mild conditions such as entropy expansivity which implies upper semi-continuity of the entropy map \cite{bowen1972entropy, misiurewicz1976topological} or $C^\infty$-smoothness of the system \cite{newhouse1989continuity}, other properties of equilibrium states are harder to come by.

Even when uniqueness is guaranteed, the equilibrium states may not have strong mixing properties as in Proposition \ref{prop: bowen thm}. For instance, Bowen has shown the following theorem which guarantees uniqueness of equilibrium states, but stops short of showing the Bernoulli property.

\begin{prop}\cite{bowen1974some}\label{prop: bowen general}
Let $f : X\to X$ be expansive and have specification, and suppose that $\vp: X\to\R$ has the Bowen property, that is, for all $\ep > 0$, there exists $K$ such that for all $n\in\N$ and $x\in X$,
$$\sup\left\{\left|\sum_{i=0}^{n-1}\vp(f^ix) - \vp(f^iy)\right| : d_n(x,y)\leq \ep \right\}\leq K.$$
Then there is a unique equilibrium state for $\vp$.
\end{prop}

Ledrappier then showed that these equilibrium states have the $K$-property by means of the following proposition.
\begin{prop}\cite[Proposition 1.4]{Ledrappier1977K}\label{prop: Led K}
Let $(X,f)$ be asymptotically entropy expansive and let $\vp : X\to\R$ be continuous. Suppose that $(X\times X,f\times f)$ has a unique equilibrium state for the potential $\Phi(x,y) = \vp(x) + \vp(y)$. Then the unique equilibrium state for $\vp$ has the $K$-property.
\end{prop}

We now present some definitions and results in subadditive thermodynamic formalism that are already known. Consider a sequence of continuous functions $\Phi =\{\log\vp_n\}_{n\in\N}$ on $(X,f)$. We say $\Phi$ is \textit{subadditive} if
\begin{equation}\label{eq: subadditive 2}
\log\vp_{m+n} \leq \log \vp_m +\log\vp_n \circ f^m
\end{equation}
for all $m,n \in \N$. 

A subadditive potential is a natural generalization of the Birkhoff sum of an additive potential in the following sense: given a continuous potential $\vp \colon X \to \R$ and denoting its $n$-th Birkhoff sum by $S_n\vp$, we obtain an equality in \eqref{eq: subadditive 2} if we replace each $\log\vp_n$ by $S_n\vp$.

Then following the definition of \cite{cao2008thermodynamic}, we define the \textit{topological pressure} of $\Phi$ as
$$P(\Phi) = \lim_{\ep\to 0}\limsup_{n\to\infty}\frac{1}{n}\log\sup\Big\{\sum_{x\in E}\vp_n(x)\mid E \subset X \text{ is } (n,\ep) \text{-separated}\Big\}.$$
The convergence of the limit is guaranteed by the subadditivity of $\Phi$. 

There is another definition for the subadditive pressure introduced by Barreira \cite{barreira1996non} using open covers. While it is not known whether two notions of the subadditive pressure coincide in the most general situations, they are shown to be equal when the entropy map $\mu \mapsto h_\mu(f)$ is upper semi-continuous \cite{cao2008thermodynamic}, which is true in our setting.

As mentioned in the introduction, it was shown in \cite{cao2008thermodynamic} that the subadditive pressure satisfies the following variational principle: if $(X,f)$ has finite topological entropy, then
$$
P(\Phi) = \sup\limits_{\mu \in \M(f)} P_\mu(\Phi)
$$
where $\displaystyle P_\mu(\Phi):=h_\mu(\s) + \lim\limits_{n \to \infty}\frac{1}{n} \int \log\vp_n\,d\mu.$ Similar to additive potentials, the existence of equilibrium states for subadditive potentials can be guaranteed by mild conditions on the base such as the upper semi-continuity of the entropy map. However, the questions on uniqueness and mixing properties are more subtle.

We also will make use of the subadditive versions of the Gibbs property as well as the bounded distortion property. A probability measure $\mu$ on $X$ has the \textit{subadditive Gibbs property} with respect to $\Phi$ if for any $\ep>0$ there exists $C \geq 1 $ such that for all $x\in X$ and $n\geq 0$,
	\begin{equation}\label{eq: Gibbs 2}
C^{-1} \leq \frac{\mu(B_n(x,\ep))}{e^{-nP(\Phi)}\vp_n(x)}\leq C.
\end{equation}
If just the lower inequality holds, we say that $\Phi$ has the \textit{lower subadditive Gibbs property}.
A subadditive potential $\Phi = \{\log\vp_n\}$ on $X$ has \textit{bounded distortion} if there exists $C \geq 1$ such that for all $\ep > 0$ sufficiently small, $x\in X$, $n\in \N$, and $y,z\in B_n(x,\ep)$, we have	
\begin{equation}\label{eq: bdd distortion}
C^{-1} \leq \frac{\vp_n(y)}{\vp_n(z)}\leq C.
\end{equation}

\begin{rem}
This is a stronger property than that referred to in \cite{cao2008thermodynamic}, and should be thought of as analogous to the Bowen property for additive potentials in Proposition \ref{prop: bowen general}
\end{rem}

While Bowen's theorem \ref{prop: bowen general} does not extend directly to general subadditive potentials, it does generalize to quasi-multiplicative subadditive potentials over uniformly hyperbolic base dynamics:

\begin{prop}\cite[Theorem 5.5]{feng2011equilibrium}\label{prop: unique eq state}
Let $\Phi = \{\log \vp_n\}_{n \in \N}$ be a subadditive potential over $(\Sig,\s)$ with bounded distortion. If $\Phi$ is quasi-multiplicative, then $\Phi$ has a unique equilibrium state $\mu \in \M(\s)$. Moreover, $\mu$ has the subadditive Gibbs property: there exists $C \geq 1$ such that for any $n\in \N$, $\I \in \L(n)$, and $x \in [\I]$, 
\begin{equation}\label{eq: Gibbs}
C^{-1}\leq \frac{\mu([\I])}{e^{-nP(\Phi)}\vp_n(x)} \leq C.
\end{equation}
\end{prop}

We note that the original setting of \cite[Theorem 5.5]{feng2011equilibrium} deals with quasi-multiplicative functions on the set of admissible words $\L$, while Proposition \ref{prop: unique eq state} considers more general subadditive potentials. However, such a generalization is rather trivial using the bounded distortion property: we may treat $\Phi$ like a function on $\L$ via \eqref{eq: wt Phi} and apply the result of \cite{feng2011equilibrium}. Another such instance can be found in the proof of Lemma \ref{lem: total ergodicity to mixing}. Moreover, the subadditive Gibbs property of $\mu_\A$ from Proposition \ref{prop: unique eq state} will play a crucial role in establishing our results.
 
\subsection{Singular value potentials and previously known results}
In this subsection, we introduce a specific class of subadditive potentials known as the singular value potentials arising from matrix cocycles. Since all of our results deal with the singular value potentials over subshifts of finite type $(\Sig,\s)$, we assume that our base dynamic is $(\Sig,\s)$ throughout the subsection.

The \textit{singular values} of $A \in M_{d \times d}(\R)$ are eigenvalues of $\sqrt{A^*A}$. We define the \textit{singular value function} $\vps \colon \Md \to \R$ with parameter $s\geq 0 $ as follows:
$$\vps(A) = \begin{cases} 
      \alpha_1(A)\ldots\alpha_{\lfloor s \rfloor}(A)\alpha_{\lceil s \rceil}(A)^{\{s\}} & 0\leq s \leq d ,\\
      |\det(A)|^{s/d} & s>d,
   \end{cases}$$
where $\alpha_1(A) \geq \ldots \geq \alpha_d(A) \geq 0$ are the singular values of $A$. 
The function $(A,s) \mapsto \vps(A)$ is upper semi-continuous, and has a discontinuity at $s=k \in \N$ only if there is a jump in the singular values of the form $\alpha_{k-1}(A)>\alpha_k(A) = 0$. In particular, if $A$ takes values in $\glr$, then $\vps(A)$ is continuous in both $A$ and $s$. 

For each $s \in \R^+$, the \textit{singular value potential} is defined by $\Phi_\A^s := \{\log \vps_{\A,n}\}_{n \in \N}$ where
$$\vps_{\A,n}(x):=\vps(\A^n(x)).$$
As $\vps$ is submultiplicative for all $s$, it follows that $\Phi_\A^s$ is a subadditive potential on $\Sig$. In the case where $s=1$, $\Phi_\A^1$ coincides with the norm potential $\Phi_\A$ introduced in the introduction. We end this section with a few remarks on the singular value potentials.

\begin{rem}\label{rem: bdd distortion}
The bounded distortion \eqref{eq: bdd distortion} holds for all norm potentials $\Phi_\A$ and singular value potentials $\Phi_\A^s$ considered in our results. If $\A$ is locally constant, then we may take $C=1$ from \eqref{eq: bdd distortion}. For fiber-bunched cocycles $\A \in C^\alpha_b(\Sig,\glr)$, the bounded distortion follows from the \hol continuity of the canonical holonomies \eqref{eq: hol holder}.
\end{rem}

\begin{rem}\label{rem: thermo conj inv}
We observe that equilibrium states are preserved under a continuous conjugacy. More specifically, we say $\A \in  C(\Sig,\glr)$ is continuously conjugated to another $\glr$-cocycle $\B$ if there exists $\CC \in C(\Sig,\glr)$ such that $\B(x) = \CC^{-1}(\s x)\A(x)\CC(x)$. This follows from the subadditive variational principle \eqref{eq: var prin} and the fact that the norm $\|\CC(x)\|$ is uniformly bounded from the compactness of $\Sig$.
\end{rem}

\section{Subadditive Thermodynamic Formalism and the $K$-property}\label{sec: K}

Many of the techniques and results in thermodynamic formalism necessary for our proof of Theorem \ref{thm: A} hold in general settings; as such, we set them apart here. 
We first show that under some general conditions, ergodicity and the Gibbs property of an equilibrium state implies uniqueness. We then establish Ledrappier's criterion in the subadditive setting. These combine to prove a general result which shows that in our setting, weak mixing is equivalent to the $K$-property. Throughout this section, we will consider $(X,f)$ to be an expansive homeomorphism on a compact metric space.

\subsection{Uniqueness of Equilibrium States}
In this subsection, we establish sufficient conditions for subadditive equilibrium states to be unique, based on \cite{bowen1974some}.
In doing so, we will need to make use of the Kolmogorov-Sinai entropy of a transformation.
For any measure $\nu$ on $X$ and any finite partition $\xi$ of $X$, define
$$H_\nu(\xi) = -\sum_{A\in\xi}\nu(A)\log\nu(A)$$
and
\begin{equation}\label{eq: entropy wrt partition}
 h_\nu(f,\xi):=\lim\limits_{n\to\infty}\frac{1}{n}H_\nu(\bigvee_{i=0}^{n-1}f^{-i}\xi)=\inf\limits_{n\to\infty}\frac{1}{n}H_\nu(\bigvee_{i=0}^{n-1}f^{-i}\xi)
\end{equation}
where the infimum is due to subadditivity. Then the Kolmogorov-Sinai entropy of $\nu$ is defined by 
$$h_\nu(f) = \sup_{\text{finite partitions }\xi} h_\nu(f,\xi).$$
By Sinai's theorem, if a partition $\xi$ generates the Borel $\sigma$-algebra, then $h_\nu(f) = h_\nu(f,\xi)$.

\begin{lem}\label{lem: subadditive Bowen lemma 8} 
	Let $\Phi = \{\log \vp_n\}_{n \in \N}$ be a subadditive potential on $X$ with bounded distortion as in \eqref{eq: bdd distortion} and suppose $\eta \in \M(f)$ is an ergodic equilibrium state of $\Phi$ with the subadditive lower Gibbs property \eqref{eq: Gibbs 2}. Then $\eta$ is the unique equilibrium state of $\Phi$. 
\end{lem}
\begin{proof}
	We follow the proof of \cite[Lemma 8]{bowen1974some} closely. Assume for the sake of contradiction that $\nu\in \M(f)$ is an ergodic equilibrium state not equal to $\eta$. Then $\nu$ and $\eta$ are mutually singular, and so there exists a $(\nu+\eta)$-measurable set $B \subset X$ such that $f(B) = B$, $\eta(B) = 0$ and $\nu(B)=1$. For instance, we could take $B$ to be the set of generic points for $\nu$.
	
	Let $4\ep>0$ be smaller than the expansivity constant of $(X,f)$ and small enough for bounded distortion \eqref{eq: bdd distortion} to hold. For each $n\in \N$ we fix a maximal $(n,2\ep)$-separated set $E_n \subset X$. Then we fix an adapted partition $\xi_n:=\{A_x \colon x \in E_n\}$ of $X$ such that $B_n(x,\ep)\subseteq A_x \subseteq \overline{B_n(x,2\ep)}$ for each $x \in E_n$. 
	
In order to make use of the expansivity assumption, define for all $n$, the partition $\Omega_n := f^{[n/2]}\xi_n$ and denote the element of $\Omega_n$ containing $y \in X$ by $\omega_n(y)$. From the construction of $\Omega_n$, for any $y \in X$ there exists some $x \in E_n$ such that $B_{n}(x,\ep) \subseteq f^{-[n/2]}\omega_n(y) \subseteq B_n(x,2\ep)$. It then follows that $f^{-[n/2]}\omega_n(y) \subseteq B_n(y,4\ep)$. Therefore expansivity gives $\bigcap\limits_{n\in \N} \omega_n(y) = \{y\}$ for all $y\in X$, and by \cite[Lemma 5.10]{CT} there exists a sequence $\{C_n\}_{n\in \N}$ where $C_n$ is a union of elements of $\Omega_n$ such that $\lim\limits_{n\to \infty} (\nu + \eta)(C_n\vartriangle B)\to 0$. Since $B$ is $f$-invariant, setting 
$$\U_n:=f^{-[n/2]}C_n \subseteq \xi_n,$$
 we have $(\nu+\eta)(\U_n \vartriangle B) \to 0$. From the assumptions on $B$, this is equivalent to $\eta(\U_n)\to 0$ and $\nu(\U_n)\to 1$.
	
As $(X,f)$ is expansive, $\xi_n$ is a generator under $f^n$ by observing that given $y,z\in \bigcap\limits_{k\in\Z} f^{kn}B_n(x_k,2\ep)$ for some $\{x_k\}_{k\in \Z} \subset X$, we have that $d(f^ky,f^kz) \leq 4\ep$ for all $k\in\Z$. Consequently,
$$nh_{\nu}(f) = h_{\nu}(f^n) = h_{\nu}(f^n,\xi_n) \leq H_{\nu}(\xi_n)$$
where the last inequality is from \eqref{eq: entropy wrt partition}.
Moreover, from the subadditivity of $\Phi$, we have
$$\lim\limits_{k \to \infty}\frac{1}{k}\int \log \vp_k \,d\nu = \inf\limits_{k \to \infty} \frac{1}{k}\int \log\vp_n \,d\nu \leq \frac{1}{n}\int \log\vp_n\,d\nu$$
for each $n \in \N$. Hence, 
	\begin{align*}
	nP(\Phi) &=n\Big(h_\nu(f)+\lim\limits_{k \to \infty} \frac{1}{k}\int\log \vp_k \,d\nu\Big)\\
	&\leq H_{\nu}(\xi_n)+\int \log \vp_n \,d\nu\\
	&= \sum\limits_{A_x \in \xi_n} \Big(-\nu(A_x) \log\nu(A_x)+\int \log \vp_n \cdot \chi_{A_x}\,d\nu\Big).
	\end{align*}
	
Let $C$ be the constant given by the bounded distortion \eqref{eq: bdd distortion} on $\Phi$. Then 
$$\int \log \vp_n \cdot \chi_{A_x}\,d\nu \leq \nu(A_x)\big(C+\log\vp_n(x)\big)$$ for all $n$ sufficiently large. In particular, we have
	$$nP(\Phi) \leq C+ \sum\limits_{A_x \in \mathcal{U}_n} \nu(A_x)\Big(- \log\nu(A_x)+ \log\vp_n(x)\Big)+\sum\limits_{A_x \cap \U_n = \emptyset} \nu(A_x)\Big(- \log\nu(A_x)+ \log\vp_n(x)\Big).$$ 
	Applying a Jensen-type inequality (see \cite[Lemma 7]{bowen1974some}) to each sum, we have 
	$$nP(\Phi)-C\leq 2C^*+\nu(\U_n)\log \left(\sum_{A_x \in \mathcal{U}_n} \vp_n(x)\right)+\nu(\U_n^c)\log \left(\sum_{A_x \cap \U_n =\emptyset} \vp_n(x)\right),$$
	where $C^*:=\max\limits_{t \in [0,1]}-t\log t$. 
	
	Let $C_0$ be the constant from the subadditive lower Gibbs property \eqref{eq: Gibbs 2} of $\eta$. Then after rearranging the terms, we have
	\begin{align*}
	-2C^*-C &\leq \nu(\U_n) \left(\log \sum_{A_x \in \U_n} \vp_n(x)e^{-nP(\Phi)}\right)+\nu(\U_n^c) \log \left(\sum_{A_x \cap \U_n = \emptyset} \vp_n(x)e^{-nP(\Phi)}\right)\\
	&\leq \nu(\U_n)\log (C_0\eta(\U_n))+\nu(\U_n^c)\log (C_0\eta(\U_n^c))\\
	&=\log C_0 +\nu(\U_n)\log \eta(\U_n)+\nu(\U_n^c)\log \eta(\U_n^c).
	\end{align*}
	This, however, is a contradiction because as we send $n\to \infty$, the lower bound $-2C^*-C$ is independent of $n\in \N$ while
	$\nu(\U_n)\log \eta(\U_n) \to -\infty$ and $\nu(\U_n^c)\log \eta(\U_n^c) \to 0$. Hence, $\nu$ cannot be an equilibrium state of $\Phi$.
\end{proof}

\subsection{Subadditive generalization of Ledrappier's criterion}
\begin{lem}
For any subadditive potential $\Phi=\{\log \vp_n\}_{n \in \N}$ on $(X,f)$, consider a sequence of continuous functions $\Psi=\{\log \psi_n\}_{n\in \N}$ on $(X\times X, f\times f)$ defined by 
\begin{equation}\label{eq: Psi}
\psi_n(x,y) := \vp_n(x) \cdot \vp_n(y).
\end{equation} 
Then $\Psi$ is subadditive and $P(\Psi) = 2P(\Phi)$.
\end{lem}
\begin{proof}
Subadditivity of $\Psi$ follows immediately: as for all $n,m\in\N$,
\begin{align*}
\log\psi_{n+m}(x,y) &= \log \vp_{n+m}(x) + \log \vp_{n+m}(y)
\\
&\leq \log \vp_n(x) + \log\vp_n\circ f^m(x) + \log\vp_n(y) + \log\vp_n\circ f^m(y)
\\
&= \log \psi_n(x,y) + \log \psi_n(f^mx,f^my).
\end{align*}

For the second statement, let $\mu$ be an equilibrium state for $\Phi$. Then $\mu\times\mu \in \M(f\times f)$, and we also have
	$$P_{\mu\times\mu}(\Psi) = h_{\mu\times\mu}(f\times f) + \lim\limits_{n\to \infty}\frac{1}{n}\int \log\psi_n\,d\mu\times\mu = 2h_\mu(f) + 2\lim\limits_{n\to \infty}\frac{1}{n}\int\log\vp_n\,d\mu = 2P_\mu(\Phi).$$
	Therefore, by the variational principle \eqref{eq: var prin}, we see that $P(\Psi) \geq 2P(\Phi)$. 
	
	For the reverse direction, we again proceed by the variational principle. Let $\nu\in\M(f\times f)$ be arbitrary, and write $\nu_1$ and $\nu_2$ to be the projections of $\nu$ onto the first and second coordinate, respectively. Each $\nu_i$ is a $f$-invariant measure on $X$. An elementary calculation shows that $h_\nu(f\times f) \leq h_{\nu_1}(f) + h_{\nu_2}(f)$ (see for instance, \cite[Fact 4.4.3]{downarowicz}), and
	$$\lim\limits_{n\to \infty} \frac{1}{n}\int\log\psi_n\,d\nu = \lim \limits_{n\to \infty}\frac{1}{n}\Big(\int\log\vp_n\,d\nu_1 + \int\log\vp_n\,d\nu_2\Big).$$
	Therefore, 
	$$P_\nu(\Psi) \leq h_{\nu_1}(f) + \lim\limits_{n\to \infty}\frac{1}{n}\int\log\vp_n\,d\nu_1 + h_{\nu_2}(f) + \lim\limits_{n\to \infty}\frac{1}{n}\int\log\vp_n\,d\nu_2 \leq 2P(\Phi).$$
\end{proof}
We immediately have the following corollary:
\begin{cor}
If $\mu\in\M(f)$ is an equilibrium state for $\Phi$, then $\mu\times\mu \in \M(f\times f)$ is an equilibrium state for $\Psi$.
\end{cor}

We can now state the subadditive generalization of Proposition \ref{prop: Led K} for establishing the $K$-property. Recall that we call a measure $\mu$ is $K$ if and only if it has no nontrivial zero entropy factors. Equivalently, the maximal zero entropy factor, called the \textit{Pinsker factor}, is trivial.

\begin{prop}\label{prop: Ledrappier subadditive}
	Let $\Phi = \{\log\vp_n\}_{n \in \N}$ be a subadditive potential on $X$ with unique equilibrum state $\mu \in \M(f)$. If $\mu \times \mu \in \M(f \times f)$ is the unique equilibrium state for $\Psi$, then $\mu$ has the $K$-property.
\end{prop}

\begin{proof}
We follow the original proof of Ledrappier, and prove the contrapositive. Let $\mu\in \M(f)$ be the unique equilibrium state for $\Phi$, and suppose it is not $K$. Then the Pinsker factor $\Pi$ for $\mu$ is non-trivial. We therefore can define $m\in\M(f\times f)$ different from $\mu\times \mu$ to be 
$$m(A\times A') = \int_A \mathbb{E}[\chi_{A'}\mid \Pi]\,d\mu$$
for all measurable $A,A'\subset X$. To see this is different from $\mu\times\mu$, take $A$ to be $\Pi$-measurable, and observe that $m(A\times A) = \mu(A)\neq \mu(A)^2 =  (\mu\times\mu)(A \times A)$. For those familiar with joinings, this is the relatively independent self-joining of $\mu$ over $\Pi$.

The entropy calculation from \cite{Ledrappier1977K} is purely dependent on the measure, and so is unaffected by the subadditive setting. For a reference where this calculation is carried out in full, see \cite{call2020}. Hence, $h_m(f\times f) = 2h_\mu(f)$. Now because $m(A\times X) = m(X\times A) = \mu(A)$, and $\psi$ is is defined independently in each coordinate, we observe that for all $n \in \N$, $\displaystyle\int \log\psi_n\,dm = 2\int\log\vp_n\,d\mu$. Therefore, 
	$$P_m(\Psi) = h_m(f\times f) + \lim\limits_{n\to \infty} \frac{1}{n}\int\log\psi_n\,dm = 2h_\mu(f) + 2\lim\limits_{n\to \infty}\frac{1}{n} \int\log\vp_n\,d\mu = 2P_\mu(\Phi) = 2P(\Phi).$$
Hence, $m$ is an equilibrium state for $\Psi$ in $\M( f\times f)$, as is $\mu\times\mu$. So there exist multiple equilibrium states for the product system.
\end{proof}

Now, recall from Proposition \ref{prop: weak mixing iff} that a measure $\mu \in \M(f)$ is weak mixing if and only if $\mu\times\mu \in \M(f \times f)$ is ergodic. Using this fact, we obtain the following theorem:

\begin{thm}\label{thm: Subadditive Kprop result}
Let $(X,f)$ be an expansive homeomorphism on a compact metric space and $\Phi = \{\log \vp_n\}_{n \in \N}$ be a subadditive potential on $X$ with bounded distortion \eqref{eq: bdd distortion}. Suppose $\eta \in \M(f)$ is a weak mixing equilibrium state of $\Phi$ with the lower subadditive Gibbs property \eqref{eq: Gibbs 2}. Then $\eta$ has the $K$-property.
\end{thm}

\begin{proof}
First, as $\eta$ is a weak mixing equilibrium state, $\eta\times\eta$ is an ergodic equilibrium state. Therefore, if we can show Lemma \ref{lem: subadditive Bowen lemma 8} holds for the system $(X\times X,f\times f)$ with potential $\Psi$ defined as \eqref{eq: Psi}, then it follows that $\eta\times\eta$ is the unique equilibrium state. Therefore, by the subadditive version of Ledrappier's criterion, it immediately follows that $\eta$ is $K$. 

We now verify the assumptions in Lemma \ref{lem: subadditive Bowen lemma 8}. First, $(X\times X,f\times f)$ is still an expansive homeomorphism on a compact metric space. Thus, we only need to check that $\Psi$ has the bounded distortion and the subadditive Gibbs property. 

Since the metric on our product space is the maximum of the distance in each coordinate \eqref{eq: product metric}, it follows that 
$$B_n((x,y),\ep) = B_n(x,\ep)\times B_n(y,\ep).$$
From this, it follows that the subadditive Gibbs property on $\eta$ and the bounded distortion of $\Phi$ induce the corresponding properties on $\eta \times \eta$ and $\Psi$.
\end{proof}

We note that weak mixing is a natural assumption to impose in this theorem, as one can easily define a system which is not weak mixing and satisfies all other conditions of this theorem.

\subsection{Relationship between one and two-sided results}\label{sec: two sided}

Many of the results that we cite (see for instance, \cite{morris2018ergodic, morris2019necessary, feng2009lyapunov}) are written in the case where the base dynamic is a one-sided shift. While it is not difficult to see that those that we cite hold in the invertible setting as well, for completeness, we sketch some of the arguments here. Throughout, let $(\Sig^+,\s)$ be a mixing one-sided shift of finite type, and let $(\Sig,\s)$ be its natural extension, which is a two-sided subshift of finite type with the same list of forbidden words. Finally, let $\pi : \Sig \to\Sig^+$ be the standard projection map, taking $(x_i)_{i\in\Z}$ to $(x_i)_{i\in\N_0}$.

The following proposition is a consequence of the definition for the subadditive pressure.
\begin{prop}
Let $\Phi$ be a subadditive potential on $(\Sig^+,\sigma)$, and consider the subadditive potential $\Psi := \Phi\circ\pi$ on $(\Sig,\s)$. Then $P(\Phi) =	P(\Psi)$.
\end{prop}


Since $\Psi$ is defined as $\Phi\circ \pi$ and the entropy is preserved under the natural extension, the following corollary is immediate from the above proposition.
\begin{cor}
Let $\mu$ be an equilibrium state for $(\Sig^+,\s)$ and $\Phi$. Then the natural extension of $\mu$ is an equilibrium state for $(\Sig,\s)$ and $\Psi$.
\end{cor}


\begin{prop}
There is a unique equilibrium state for a subadditive potential $\Phi$ on $(\Sig^+,\s)$, if and only if its natural extension is the unique equilibrium state for $\Psi$ on $(\Sig,\s)$.
\end{prop}

\begin{proof}
It suffices to show that different equilibrium states for the natural extension project to different equilibrium states for the one-sided system. That they project to different measures follows from shift-invariance and the fact that they must differ on some cylinder set. That they project to equilibrium states follows because 
$\Psi$ is defined as $\Phi \circ \pi$ and the entropy is preserved under the natural extension.
\end{proof}

\begin{cor}\label{natural extension}
For any locally constant cocycle $\A$ and $s > 0$, the natural extension of any equilibrium state for $\Phi_{\A}$ over $(\Sig^+,\s)$ is an equilibrium state for the invertible setting with the same potential.
\end{cor}

Using the classical fact that mixing and ergodicity of natural extensions are equivalent to the respective properties for the one-sided systems, we see that Proposition \ref{prop: Morris equivalence} holds for two-sided shifts as well.

\section{Proof of main theorems}\label{sec: proof of main thms}
\subsection{Proof of Theorem \ref{thm: A}}
We recall the setting of Theorem A. Let $(\Sigma_T,\s)$ be a mixing subshift of finite type and $\Phi = \{\log \vp_n\}_{n \in \N}$ be a quasi-multiplicative subadditive potential with bounded distortion. Let $\mu\in\M(\s)$ be the unique equilibrium state for $\Phi$ with the Gibbs property from Proposition \ref{prop: unique eq state}, and suppose that $\mu$ is totally ergodic. We wish to show that $\mu$ is $K$. By Theorem \ref{thm: Subadditive Kprop result}, it suffices to show that $\mu$ is weak mixing.

The following proposition is essentially a reformulation of \cite[Theorem 5 (ii)]{morris2018ergodic}. The setting there is for norm potentials of irreducible locally constant cocycles; however, the proof generalizes easily to any quasi-multiplicative subadditive potentials with bounded distortion.

\begin{lem}\label{lem: total ergodicity to mixing}
Let $\Phi = \{\log \vp_n\}_{n \in \N}$ be a quasi-multiplicative subadditive potential on $\Sig$ with bounded distortion. Suppose the unique equilibrium state $\mu \in \M(\s)$ of $\Phi$ from Proposition \ref{prop: unique eq state} is totally ergodic. Then $\mu$ is mixing.
\end{lem}
\begin{proof}
The proof of \cite[Theorem 5 (ii)]{morris2018ergodic} extends without much modification; we only point out minor modifications required to extend the proof. 

From Proposition \ref{prop: unique eq state}, it follows that $\mu$ has the Gibbs property \eqref{eq: Gibbs} with constant $C_0$. Recalling the notation $\wt{\Phi}$ from \eqref{eq: wt Phi} and denoting the constant from the bounded distortion \eqref{eq: bdd distortion} of $\Phi$ by $C_1$, for any $n \in \N$, $\I \in \L(n)$, and $x \in [\I]$ we have the following bounds on $\mu([\I])/\big(e^{-nP(\Phi)}\wt{\Phi}(\I)\big)$:  
$$(C_0C_1)^{-1} \leq C_1^{-1}\cdot  \frac{\mu([\I])}{e^{-nP(\Phi)}\vp_n(x)} \leq \frac{\mu([\I])}{e^{-nP(\Phi)}\wt{\Phi}(\I)} \leq C_0$$
Then for any cylinders $\I,\J \in \L$ of length $n$ and $m$, we have for any $k>n$,
\begin{align*}
\mu([\I] \cap f^{-k}[\J]) &=\sum\limits_{\substack{|\K| = k-n \\ \I\K\J \in \L}} \mu([\I\K\J])\\
&\leq C_0\sum\limits_{\substack{|\K| = k-n \\ \I\K\J \in \L}}e^{-(k+m)P(\Phi)}\wt{\Phi}(\I\K\J)\\
&\leq C_0\sum\limits_{\substack{|\K| = k-n \\ \I\K\J \in \L}}e^{-(k+m)P(\Phi)}\wt{\Phi}(\I)\wt{\Phi}(\K)\wt{\Phi}(\J)\\
&\leq C_0^4C_1^3 \mu([\I]) \mu([\J]) \Big( \sum\limits_{\substack{|\K| = k-n \\ \I\K\J \in \L}} \mu([\K])\Big)\\
&\leq C_0^4C_1^3 \mu([\I]) \mu([\J]).
\end{align*}
This gives $\limsup\limits_{k \to \infty}\mu([\I] \cap f^{-k}[\J]) \leq C \mu([\I])\mu([\J])$ where $C=C_0^4C_1^3$. Using this property together with total ergodicity of $\mu$, the rest of the proof from here on (i.e., promoting total ergodicity to weak mixing, and then to mixing) follows that of \cite[Theorem 5 (ii)]{morris2018ergodic} verbatim, following the method of Ornstein \cite{Ornstein72}.
\end{proof}

Theorem \ref{thm: A} now follows as an easy consequence of Proposition \ref{prop: unique eq state} which gives an equilibrium state with the Gibbs property, Lemma \ref{lem: total ergodicity to mixing} which shows that total ergodicity is enough to get weak mixing, and Theorem \ref{thm: Subadditive Kprop result}, which lifts these together to $K$.

\subsection{Proof of Theorem  \ref{thm: B}}

%

Suppose $\A \colon \Sigma \to \glr$ is a locally constant cocycle, and its image consists of matrices $\{A_1,\ldots,A_q\}$. We say $\A$ is \textit{irreducible} if there does not exist a proper subspace $V \subset \R^d$ with $A_iV = V$ for every $1 \leq i \leq q$.

The following result of \cite{feng2009lyapunov} guarantees the quasi-multiplicativity of the unique equilibrium state of $\Phi_\A$ under irreducibility:
\begin{prop}\cite[Proposition 2.8]{feng2009lyapunov}\label{prop: irred implies qm}
	Let $\A \colon \Sigma \to \glr$ be a locally constant cocycle. If $\A$ is irreducible, then the norm potential $\Phi_\A$ is quasi-multiplicative and has a unique equilibrium state $\mu_\A \in \M(\s)$. 
\end{prop}

Then Theorem \ref{thm: B} is a direct consequence of Theorem \ref{thm: A} and Proposition \ref{prop: irred implies qm}:

\begin{proof}[Proof of Theorem \ref{thm: B}]
	The norm potentials $\Phi_\A$ of any locally constant cocycles $\A \colon \Sigma \to \glr$ immediately have bounded distortion with constant $C=1$. Irreducibility of $\A$ then gives quasi-multiplicativity of $\Phi_\A$ by Proposition \ref{prop: irred implies qm}. Hence, if $\mu_\A \in \M(\s)$ satisfies either of the equivalent conditions in Proposition \ref{prop: Morris equivalence}, then $\mu_\A$ is mixing which is stronger than total ergodicity, and it then follows from Theorem \ref{thm: A} that $\mu_\A$ has the $K$-property.
\end{proof}

We end this subsection with a few related remarks. First, we comment on the difference between norm potentials considered in Theorem \ref{thm: B} with similar subadditive potentials considered by Morris.
\begin{rem}\label{rem: other potentials}
	Morris \cite{morris2018ergodic, morris2019necessary} works with similar subadditive potentials. He considers locally constant cocycles $\A$ over one-sided full shifts $(\Sigma^+,\s)$, and defines the subadditive pressure of $(\A,s)$ as
	$$P(\A,s)=\lim\limits_{n \to \infty}\frac{1}{n}\log\Big(\sum\limits_{i_1,\ldots,i_n = 1}^q \|A_{i_n}\ldots A_{i_1}\|^s \Big).$$
	Making this explicit, for all $s > 0$, the corresponding subadditive potential is given by $\Psi_\A^s = \{\log(\psi_n)^s\}$, where $\psi_n(x) = \|\A^n(x)\|$. Thus, we see that the subadditive potential $(\A,1)$ in his papers agrees with the norm potential $\Phi_\A$, and in fact, $(\A,s)$ agrees with the singular value potentials $\Phi_\A^s$ when $s\leq 1$. This is not necessarily so when $s > 1$. Nevertheless, Theorem \ref{thm: A} applies to Morris' subadditive potentials $(\A,s)$ for all $s$. If $\A$ is an irreducible locally constant cocycle, then by Proposition \ref{prop: irred implies qm}, $\Psi_\A^s$ is quasi-multiplicative, and the bounded distortion condition is satisfied trivially because it is locally constant. Therefore, if the unique equilibrium state $\mu$ of $(\A,s)$ is totally ergodic, then it is $K$, establishing partial progress towards Conjecture 2 of \cite{morris2018ergodic}.
\end{rem}

\begin{rem}
For irreducible locally constant cocycles, their unique equilibrium states are shown to possess various ergodic properties depending on the assumptions. Theorem \ref{thm: B} is one result in this direction. Another such instance can be found in \cite{piraino2018weak}; Piraino established that under primivity or strong irreduciblity and proximality, the unique equilibrium states are weakly Bernoulli, and so are isomorphic to a Bernoulli shifts. On the other hand, such equilibrium states are rarely Bernoulli measures. Morris and Sert \cite{morris2019converse} established a mild assumption that prevents such equilibrium states from being Bernoulli measures on full shifts $\Sigma$. We stress that both of these results lie in the setting of Theorem B, and do not apply to the setting of the following results.

\comment{
	\textcolor{red}{We provide one example where some of these results can be applied. Consider a locally constant cocycle $\A \colon \Sigma_2 \to \gltwo$ generated by 
$$A_1 = \begin{pmatrix}
2 & 0\\
0 & 1
\end{pmatrix} \text{ and }A_2 = \begin{pmatrix}
0 & 1\\
1 & 0
\end{pmatrix}.$$
It is clear that $\A$ is irreducible, hence $\Phi_\A$ has a unique equilibrium state $\mu_\A$. Moreover, direct computation shows that $\A^2$ is also irreducible, hence the unique equilibrium state for $\Phi_{\A^2}$ must coincide with $\mu_\A$. In particular, this shows that $\mu_\A$ is ergodic with respect to $\s^2$, verifying the second condition from Proposition \ref{prop: Morris equivalence}. Theorem \ref{thm: B} then applies to show that $\mu_\A$ is $K$. }

\textcolor{red}{
On the other hand, due to the presence of a proximal element $A_1$, \cite[Theorem 5]{morris2019converse} applies to show that $\mu_\A$ is not a Bernoulli measure. By not being a Bernoulli measure, we are claiming that $\mu_\A$ is not a $(p,1-p)$-measure on $\Sigma_2$, and not claiming that $(\Sig,\s,\mu_\A)$ is not measurably isomorphic to a Bernoulli shift. In fact, while \cite{piraino2018weak} does not directly apply to show that $\mu_\A$ is weakly Bernoulli (and hence isomorphic to a Bernoulli shift) as $\A$ is not strongly irreducible, it could be that $(\Sig,\s,\mu_\A)$ may end up being measurably isomorphic to a Bernoulli shift via other methods.
}}
 
\comment{
\textcolor{blue}{I agree with your comment; I think showing that Piraino's work doesn't apply is more relevant than Morris-Sert's work doesn't apply. Anyway, to address your comment, if you have a proximal element $A$, then the semigroup containing cannot be compact. And what Morris-Sert proves is really and ``if and only if'' statement: Bernoulli measure if and only if conditions of Theorem 5 holds. So having an proximal element forces the measure to be a non-Bernoulli measure. 
As for Piraino's work, I guess the example above is not applicable to Piraino's Thm 3.3, but his Theorem 1.3 may still be applicable. So I guess we have to find an example that is neither primitive nor strongly irred and proximal. I will think about this today.} 

\textcolor{blue}{I gave it some thoughts, and I think I know what kind of example we have to construct, except none of the obvious guesses worked. Supposing we are constructing a locally constant $\gltwo$ cocycles, here are a few things we do/don't need:
\begin{enumerate}
\item proximal element (this will make our job easier and its adding this in is not a big restriction)
\item irreducibility (or quasi-multiplicativity for uniqueness of $\mu$)
\item $\mu$ is $\s^2$-ergodic (for K)
\item not primitive 
\item not strongly irred.
\end{enumerate}
If we can cook up such example, then our theorem will be applicable  but not Piraino's.
``Not strongly irreducibility'' means that our cocycle has to preserve a union of two lines. This forces our cocycles to be generated by diagonal and off diagonal matrices. But then I am not quite sure if I have an example of ``not primitive'' that meets this criteria. Technically, primitivity is stronger than irreducibility (see Piraino's work), but I do not quite have an example of such satisfying ``not strongly irreduciblity.'' As a matter of fact, I do not have an example of collection of matrices that is irreducible but not primitive; roughly I know that when the cone $K$ we are working with is the first quadrant (including the axes) and the matrices have non-negative entries, then primitivity means for some $k \in \N$, $\sum_{|I|=k}A(I)$ is a positive matrix so that it maps $K$ into its interior and irreducibility means for some $k\in \N$, $\sum\limits_{0 \leq i \leq k}\sum_{|I|=i}A(I)$ is a positive matrix (or something along this line). Maybe this is where we should start, figuring out a collection of nonnegative $\gltwo$ matrices that are irred but not primitive.
}}
\end{rem}

\subsection{Proof of Theorem  \ref{thm: C}}
Let $\A \colon \Sig \to \gltwo$ be a \hol continuous and fiber-bunched cocycle. The proof of Theorem \ref{thm: C} relies on the results of \cite{butler2019thermodynamic}.
We begin by introducing the notion of irreducibility for fiber-bunched $\glr$-cocycles. 

\begin{defn}\label{defn: irred fb}
We say a fiber-bunched cocycle $\A \in C^\alpha_b(\Sig,\glr)$ is \textit{reducible} if there exists a proper $\A$-invariant and $H^{s/u}$-invariant sub-bundle. We say $\A$ is \textit{irreducible} if it is not reducible.
\end{defn}
\begin{rem}
For fiber-bunched cocycles, irreducibility is a weaker assumption than typicality because typical cocycles are necessarily irreducible. Additionally, whenever a cocycle is both locally constant and fiber-bunched, this definition of irreducibility coincides with that of the previous section.
\end{rem}

The following proposition summarizes the results in \cite{butler2019thermodynamic}:
\begin{prop}\cite{butler2019thermodynamic}\label{prop: BP}
Let $\A \in C^\alpha_b(\Sig,\gltwo)$. If $\A$ is irreducible, then $\Phi_\A$ is quasi-multiplicative, and hence, has a unique equilibrium state $\mu_\A \in \M(\s)$. 

If $\A$ is reducible, then $\A$ is \hol continuously conjugated to another $\gltwo$-cocycle $\B$ taking values in the group of upper triangular matrices:
\begin{equation}\label{eq: B}
\B(x):=\begin{pmatrix}
a(x) & b(x)\\
0 & c(x)
\end{pmatrix}.
\end{equation}
The set of ergodic equilibrium states of $\Phi_\A$ is a subset of $\{\mu_{\log|a|},\mu_{\log|c|}\}$ where $\mu_{\log|\tau|}$, $\tau \in \{a,c\}$, is the unique equilibrium state for $\log|\tau|$. 

Moreover, $\Phi_\A$ has two distinct ergodic equilibrium states if and only if
\begin{enumerate}
\item $\log|a|$ is not cohomologous to $\log|c|$, and
\item $P(\log|a|) = P(\log|c|)$.
\end{enumerate}
Otherwise, $\Phi_\A$ has a unique equilibrium state.
\end{prop}

\begin{rem}
In \cite{butler2019thermodynamic}, it was shown that $\Phi_\A$ has a unique equilibrium state if $\A$ is irreducible. It was done via the dichotomy that either $\Phi_\A$ is quasi-multiplicative or $\A$ is \hol continuously conjugated into the group of conformal linear transformations. The latter case is clearly also quasi-multiplicative.
Hence, we stated in Proposition \ref{prop: BP} that $\Phi_\A$ is quasi-multiplicative when $\A$ is irreducible.
\end{rem}

\begin{lem}\label{lem: reducible Bernoulli}
If $\A \in C^\alpha_b(\Sig,\gltwo)$ is reducible, then all ergodic equilibrium states of $\Phi_\A$ are Bernoulli.
\end{lem}
\begin{proof}
Since the set of equilibrium states of $\Phi_\A$ is a subset of $\{\mu_{\log|a|},\mu_{\log|c|}\}$ from Proposition \ref{prop: BP} and both $\mu_{\log|a|}$ and $\mu_{\log|c|}$ are Bernoulli from Proposition \ref{prop: bowen thm}, our claim follows.
\end{proof}

For each $n \in \N$, consider $\A^n$ as a cocycle over $(\Sig,\s^n)$ and denote the corresponding norm potential by $\Phi_{\A^n}$. It can be easily checked from the definition that if $\A$ is fiber-bunched over $(\Sig,\s)$, then so is $\A^n$ over $(\Sig,\s^n)$.

\begin{lem}\label{lem: nP} For any $n\in \N$, we have $P(\Phi_{\A^n}) = nP(\Phi_{\A})$. Moreover, any equilibrium state $\mu\in \M(\s)$ of $\Phi_\A$ is an equilibrium state of $\Phi_{\A^n}$.
\end{lem}
\begin{proof}
We proceed via the variational principle \eqref{eq: var prin}. Observe that $\vp_{\A^n,m}= \vp_{\A,mn}$. Then, we see that if $\mu\in\M(\s)$,
$$h_\mu(\sigma^n) + \lim\limits_{m\to\infty}\frac{1}{m}\int\log\vp_{\A^n,m}\,d\mu = nh_\mu(\sigma) + n\lim\limits_{m\to\infty}\frac{1}{mn}\int\log\vp_{\A,mn}\,d\mu.$$
Considering $\mu$ as a $\s^n$-invariant measure, we have just shown that $P_\mu(\Phi_{\A^n}) = nP_\mu(\Phi_\A)$ and the variational principle implies that $nP(\Phi_{\A}) \leq P(\Phi_{\A^n})$.

For the reverse inequality, take $\mu\in\M(\sigma^n)$ and define $\displaystyle \nu = \sum_{i=0}^{n-1}\frac{(\sigma^i)_*\mu}{n}$. Then $\nu$ is $\sigma$-invariant, and furthermore, $h_\mu(\sigma^n) = h_\nu(\sigma^n) = nh_\nu(\sigma).$
Since $\A$ is continuous and $\Sig$ is compact, for any $0 \leq i \leq n-1$, two functions $\log \vp_{\A^n,m}\circ \s^i $ and $\log \vp_{\A^n,m}$ are uniformly comparable. Hence, $$\displaystyle \lim\limits_{m\to\infty}\frac{1}{m} \int\log \vp_{\A^n,m}\circ \s^i \,d\mu= \lim\limits_{m\to\infty}\frac{1}{m} \int\log \vp_{\A^n,m}\, d\mu.$$
Then it follows that
\begin{align*}
\lim\limits_{m\to\infty}\frac{1}{mn}\int\log\vp_{\A,mn}d\nu &= \lim\limits_{m\to\infty}\frac{1}{mn}\sum_{i=0}^{n-1}\frac{1}{n}\int\log\vp_{\A^n,m}\circ\sigma^i\,d\mu
\\
&=\frac{1}{n}\lim\limits_{m\to\infty}\frac{1}{m}\int\log\vp_{\A^n,m}\,d\mu.
\end{align*}
Therefore, $P_\mu(\Phi_{\A^n}) = nP_\nu(\Phi_\A)$. This gives the reverse inequality, and the result follows. That any equilibrium state of $\Phi_\A$ is an equilibrium state of $\Phi_{\A^n}$ is now a direct consequence.
\end{proof}

We are now ready to prove Theorem \ref{thm: C}.

\begin{proof}[Proof of Theorem \ref{thm: C}]
In view of Lemma \ref{lem: reducible Bernoulli} it suffices to focus on irreducible $\gltwo$-cocycles. Let $\A \in C^\alpha_b(\Sig,\gltwo)$ be irreducible, and $\mu_\A \in \M(\s)$ be the unique equilibrium state for $\Phi_\A$ from Proposition \ref{prop: BP}.

We then consider the cocycle $\A^2$ over $(\Sig,\s^2)$. Noting that Proposition \ref{prop: BP} also applies to $\A^2$, we divide into two cases depending on the number of equilibrium states of $\Phi_{\A^2}$.
\\

\noindent\textbf{Case 1: $\Phi_{\A^2}$ has a unique equilibrium state.}

Such a unique equilibrium state must be $\mu_\A$ by Lemma \ref{lem: nP}. Uniqueness then implies that $(\Sig,\s^2,\mu_\A)$ is ergodic. We claim that in fact, $\mu_\A$ is totally ergodic, which by Theorem \ref{thm: A}, would imply that $\mu_\A$ is $K$.

Assume not for contradiction. Let $n \in \N$ be the least integer such that $\mu_\A$ is not ergodic with respect to $(\Sig,\s^n)$. As $(\Sig,\s,\mu_\A)$ and $(\Sig,\s^2,\mu_\A)$ are ergodic, we know that $n\geq 3$. Since $\A^n$ is still a fiber-bunched $\gltwo$-cocycle over $(\Sig,\s^n)$, $\Phi_{\A^n}$ has at most two distinct ergodic equilibrium states by Proposition \ref{prop: BP}. Furthermore, in the proof of \cite[Theorem 2]{morris2019necessary}, Morris showed that the number of distinct ergodic equilibrium states for $\Phi_{\A^n}$ bounds $n$. Therefore, $2\geq n\geq 3$, a contradiction.\\

\noindent\textbf{Case 2: $\Phi_{\A^2}$ has multiple equilibrium states.}
From Proposition \ref{prop: BP}, $\A^2$ over $(\Sig,\s^2)$ must be reducible and $\Phi_{\A^2}$ must have two distinct ergodic equilibrium states $\mu_1,\mu_2 \in \M(\s^2)$.
In fact, denoting the $\A^2$-invariant and $H^{s/u}$-invariant line bundle by $L_1$, consider another line bundle $L_2$ defined by $L_2(\s x):=\A(x)L_1(x)$. Since $\A$ is irreducible, $L_2$ is different from $L_1$. We then have $L_1(\s x) = \A (x) L_2(x)$ from the $\A^2$-invariance of $L_1$, and $L_2$ is also $\A^2$-invariant.

For each $x\in \Sig$, let $\CC(x) \in \gltwo$ be the unique linear map that takes the standard basis of $\R^2$ into $\{L_1(x),L_2(x)\}$. Then $\B(x):=\CC(\s x)^{-1}\A(x)\CC(x)$ exchanges the coordinate axes of $\R^2$, and hence must be of the form specified in Theorem \ref{thm: C}:
$$\B(x) = \begin{pmatrix}
0 & a(x)\\
b(x) & 0
\end{pmatrix}.$$
Then $\B^2(x)$ is the diagonal matrix given by $\text{diag}(a(\s x)b(x),a(x)b(\s x))$. Moreover, two potentials $\a(x):=\log |a(\s x)b(x)|$ and $\b(x):=\log |b(\s x)a(x)|$ have the same pressure (with respect to $\s^2$), and their $\s^2$-ergodic equilibrium states are $\mu_1$ and $\mu_2$, respectively, each of which is Bernoulli by Lemma \ref{lem: reducible Bernoulli}. From the assumption that $\mu_1$ and $\mu_2$ are distinct, we must have that $\s_*\mu_1 = \mu_2$ and $\s_*\mu_2=\mu_1$. This is because $\s_*\mu_1$ is an $\s^2$-ergodic invariant measure and an equilibrium state for $\Phi_{\A^2}$, and hence it must be either $\mu_1$ itself or $\mu_2$. However, it cannot be equal to $\mu_1$ as this would imply the $\s$-invariance of $\mu_1$, and by the uniqueness of the equilibrium state $\mu_\A$ for $\Phi_\A$, this would imply $\mu_1 = \mu_\A$. But applying the same argument to $\mu_2$ contradicts $\mu_1$ and $\mu_2$ being distinct measures.


We will now show that $\mu_\A$ is the average of $\mu_1$ and $\mu_2$. Indeed, $\displaystyle \frac{1}{2}(\mu_1+\s _*\mu_1) = \displaystyle \frac{1}{2}(\mu_1+\mu_2)$ is $\s$-invariant from the $\s^2$-invariance of $\mu_1$, and is an equilibrium state for $\Phi_\A$. Since $\mu_\A$ is the unique equilibrium state for $\Phi_\A$, we must have that $\displaystyle \mu_\A=\frac{1}{2}(\mu_1+ \mu_2)$. 
\end{proof}

\subsection{Proof of Theorem  \ref{thm: D}}
For any $\glr$-cocycle $\A$, $n \in \N$, and $s>0$, similar to the notation from the previous subsection we denote by $\Phi_{\A^n}^s$ the $s$-singular value potential of $\A^n$ with respect to $(\Sig,\s^n)$. 

The following proposition gives us that the singular value potentials of typical cocycles have unique equilibrium states:
\begin{prop}\cite[Theorem B]{park2019quasi}\label{prop: typical unique eq} Let $\A \in C^\alpha_b(\Sig,\glr)$ be an $\alpha$-\hol and fiber-bunched cocycle. If $\A$ is typical, then $\Phi_\A^s$ is quasi-multiplicative, and hence, has a unique equilibrium state $\mu_{\A,s} \in \M(\s)$ for every $s \in [0,\infty)$.
\end{prop}

In view of Theorem \ref{thm: A} and Proposition \ref{prop: typical unique eq}, the only missing ingredient in proving Theorem \ref{thm: D} is the total ergodicity of $\mu_{\A,s}$ which we establish below.
The idea of the proof is similar to that of \cite[Theorem 5 (i)]{morris2019necessary}. 
\begin{prop}\label{prop: typical then totally ergodic}
Let $\A \in C^\alpha_b(\Sig,\glr)$ be typical and $s \in [0,\infty)$. Then the unique equilibrium state $\mu_{\A,s} \in \M(\s)$ of $\Phi_\A^s$ is totally ergodic.
\end{prop}
\begin{proof}
For any $n\in \N$, $\A^n$ is typical with respect to $(\Sig,\s^n$) via the same periodic and the homoclinic points $p$ and $z$ from the definition of typical cocycles. Applying Proposition \ref{prop: typical unique eq} to $\A^n$ and $(\Sig,\s^n)$, $\Phi_{\A^n}^s$ has a unique equilibrium state $\mu_{\A^n,s} \in \M(\s^n)$. In particular, $\mu_{\A^n,s}$ is ergodic with respect to $(\Sig,\s^n)$. From Lemma \ref{lem: nP} which also applies to $\Phi_{\A^n}^s$, it follows that $\mu_{\A^n,s}$ coincides with $\mu_{\A,s}$. Hence, $\mu_{\A,s}$ is totally ergodic. 
\end{proof}

\subsection*{Acknowledgments}
We would like to thank Ian Morris, Mark Piraino, Anthony Quas, and Cagri Sert for helpful comments, as well as for pointing out an incorrect statement about Bernoulli systems in an earlier version of this paper.

\bibliographystyle{amsalpha}
\bibliography{K-property}
\end{document}